\documentclass[10pt]{amsart}

\usepackage{extsizes}
\title[Category $\mathcal{J}$ Modules for Hamiltonian Vector Fields on a Torus]{Category $\mathcal{J}$ Modules for Hamiltonian Vector Fields on a Torus}
\author{John Talboom}
\newtheorem{thm}{Theorem}[section]

\newtheorem{lem}[thm]{Lemma}
\newtheorem{prop}[thm]{Proposition}
\newtheorem{defn}[thm]{Definition}

\newtheorem*{theorem*}{Theorem}
\begin{document}
\begin{abstract}
Modules for the Lie algebra of Hamiltonian vector fields on a torus, which admit a compatible action for the commutative algebra of multivariate Laurent polynomials, and have finite dimensional weight spaces, are called category $\mathcal{J}$. This paper classifies the indecomposable and the irreducible modules in category $\mathcal{J}$.
\end{abstract}
\maketitle
Key Words: \keywords{Indecomposable representations; Irreducible representations; Lie algebra of Hamiltonian vector fields; Weight module}

2010 Mathematics Subject Classification: \subjclass{17B10, 17B66} 
\section{Introduction}
An important class of simple infinite dimensional Lie algebras are the Lie algebras of Cartan type, which arise in the study of vector fields on a manifold. These are categorized into four series, the general, special, Hamiltonian, and contact, denoted $W_N, S_N, H_N$ and $K_N$ respectively. Type $W_N$ is the Lie algebra of derivations of $R_N=\mathbb{C}[x_1,\dots,x_N]$, whose elements are vector fields of the form $X=\sum_{i=1}^Nf_i\frac{\partial}{\partial x_i}$ where $f_i\in R_N$ (some sources allow $R_N$ to be all formal power series, $\mathbb{C}[[x_1,\dots,x_N]]$). The remaining three series are defined using the Lie derivative $\mathcal{L}_X\omega$, of a specific differential form $\omega$ with respect to vector field $X$. Specifically $S_N=\{X|\mathcal{L}_X(dx_1\wedge\dots\wedge dx_N)=0\}$, $H_{2m}=\{X|\mathcal{L}_X(\sum_{i=1}^{m}dx_i\wedge dx_{m+i})=0\}$ and $K_{2m+1}=\{X|\mathcal{L}_X(\omega)=P\omega\}$, where $\omega=dx_{2m+1}+\sum_{i=1}^{m}x_{m+i}dx_i-x_i dx_{m+i}$ and $P\in R_{2m+1}$ (see \cite{Ru} for more details).

Recent works have studied the representation theory of the Lie algebra of derivations $\text{Der}(A_N)$, where $A_N=\mathbb{C}[t_1^{\pm1},\dots,t_N^{\pm1}]$. Elements of $\text{Der}(A_N)$ can be identified with the polynomial vector fields on a torus (see Section 2). An interesting class of modules, called $(A_N,\text{Der}(A_N))$-modules is considered in \cite{Rao2}, which are modules that admit a compatible action by both $\text{Der}(A_N)$ and the commutative algebra $A_N$. In said paper Eswara Rao classifies all irreducible $(A_N,\text{Der}(A_N))$-modules with finite dimensional weight spaces. These modules are presented in \cite{B} as category $\mathcal{J}$, and the classification is extended to all indecomposable modules by showing that $\text{Der}(A_N)$ has polynomial action. 

A similar result is achieved in \cite{BT}, again using the strategy of showing polynomial action, where category $\mathcal{J}$ is considered for $\mathcal{S}_N$, the divergence zero vector fields on a torus. An inductive proof is used to show that the general case follows from the case $N=2$, which turns out to be exceptional.

The current paper now considers category $\mathcal{J}$ for the Lie algebra of Hamiltonian vector fields on an $N$ dimensional torus, denoted $\mathcal{H}_N$, where $N$ is even. The goal is to classify all indecomposable and all irreducible category $\mathcal{J}$ modules. Theorem \ref{classification} is the main result and is summarized below (the action of $\mathcal{H}_N$ will be given later).
\begin{theorem*}
Let $\mathfrak{H}_N^+$ be the vector fields in $H_N$ with non-negative degree, and $\mathfrak{h}_N$ the $2m+1$ dimensional Heisenberg algebra. Let $\lambda\in\mathbb{C}^N$ and $\mathcal{J}_{\lambda}$ a subcategory of modules in $\mathcal{J}$ supported on $\lambda+\mathbb{Z}^N$, where $N=2m\geq 2$. There is an equivalence of categories between the category of finite dimensional modules for $\mathfrak{H}_N^+\oplus\mathfrak{h}_{N}$ and $\mathcal{J}_{\lambda}$. This equivalence maps $V$ to $A_N\otimes V$ where $V$ is a finite dimensional module for $\mathfrak{H}_N^+\oplus\mathfrak{h}_{N}$.
\end{theorem*}
To obtain this result the general strategy of \cite{BT} is applied. Since $\mathcal{H}_2\cong\mathcal{S}_2$, the basis of induction has already been shown (see \cite{BT} Proposition 3.7). Irreducible representations for the case $N=2$ are studied in \cite{JL} by Jiang and Lin and, although not shown here, their paper provides a crucial step in proving the $N=2$ case (see \cite{BT} Lemma 3.1). 

This paper is organized in the following way. Section 2 will present Lie algebra $\mathcal{H}_N$ and define its category $\mathcal{J}$ modules. It will be shown here that the action of $\mathcal{H}_N$ on any module from category $\mathcal{J}$ is determined completely by the action restricted to a single weight space $V$. Section 3 is main part of the proof for the classification theorem. Starting with the previously proven $N=2$ case as the basis of induction, it is shown that $\mathcal{H}_N$ acts by certain $\text{End}(V)$-valued polynomials for all $N$. In Section 4 the polynomial action of $\mathcal{H}_N$ is seen to be a representation of Lie algebras $\mathfrak{H}_N^+$  and $\mathfrak{h}_{N}$, and the main result can be stated. For simple modules of category $\mathcal{J}$ the action simplifies considerably. Section 5 is devoted to irreducible representations, and exhibits the main result when only simple modules are considered. 
\section{Preliminaries}
Let $A_N=\mathbb{C}[t_1^{\pm1},\dots,t_{N}^{\pm1}]$ be the algebra of Laurent polynomials over $\mathbb{C}$. Elements of $A_N$ are presented with multi-index notation $t^r=t_1^{r_1}\dots t_{N}^{r_{N}}$ where $r=(r_1,\dots,r_{N})\in\mathbb{Z}^{N}$. Let $\{e_1,\dots,e_N\}$ denote the standard basis for $\mathbb{Z}^N$. For $k\in\mathbb{Z}^N	$, $|k|=k_1+\dots+k_N$, $k!=k_1!\dots k_N!$ and $\binom{r}{k}=\frac{r!}{k!(r-k)!}$. Denote the set of non-negative integers by $\mathbb{Z}_{\geq0}$.

For $i\in\{1,\dots,N\}$, let $d_i=t_i\frac{\partial}{\partial t_i}$. The derivations of $A_N$, $\text{Der}(A_N)=\text{Span}_{\mathbb{C}}\left\{t^rd_i|i\in\{1,\dots,N\}, r\in\mathbb{Z}^{N}\right\}$, forms a Lie algebra called the Witt algebra denoted here by $\mathcal{W}_N$, with Lie bracket $[t^rd_i,t^sd_j]=s_it^{r+s}d_j-r_jt^{r+s}d_i$. By Setting $t_j=e^{\sqrt{-1}\theta_j}$ for all $j\in\{1,\dots,N\}$, $\mathcal{W}_N$ may be identified with the Lie algebra of (complex-valued) polynomial vector fields on an $N$ dimensional torus, where $\theta_j$ is the $j$th angular coordinate. 
 
The focus of this paper is on the subalgebra of Hamiltonian vector fields $\mathcal{H}_N$. Let $X=\sum_{i=1}^Nf_id_i\in\mathcal{W}_{N}$ and from now on let $N=2m$. The change of coordinates $t_j=e^{\sqrt{-1}\theta_j}$, gives $\frac{\partial}{\partial \theta_j}=\frac{\partial t_j}{\partial \theta_j}\cdot\frac{\partial}{\partial t_j}=\sqrt{-1} t_j\frac{\partial}{\partial t_j}=\sqrt{-1} d_j$. Thus $X$ can be expressed $X=-\sqrt{-1}\sum_{j=1}^{N}f_j(t)\frac{\partial}{\partial \theta_j}$. The defining property, $\mathcal{L}_X(\sum_{i=1}^{m}d\theta_i\wedge d\theta_{m+i})=0$, of $\mathcal{H}_N$ yields that for $i=1,\dots,m$, $f_i=\frac{\partial u}{\partial \theta_{m+i}}=\sqrt{-1} t_{m+i}\frac{\partial u}{\partial t_{m+i}}$ and $f_{m+i}=-\frac{\partial u}{\partial \theta_i}=-\sqrt{-1} t_i\frac{\partial u}{\partial t_i}$  for some $u\in A_N$, or $X=d_i$. Thus $\mathcal{H}_N$ is spanned by elements $X_u=\sum_{i=1}^mt_{m+i}\frac{\partial u}{\partial t_{m+i}}d_i-t_i\frac{\partial u}{\partial t_i}d_{m+i}$, and $d_i$. Obtain homogeneous elements by letting $u=t^r$, for $r\in\mathbb{Z}^N$, and set $h(r)=\sum_{i=1}^mr_{m+i}t^rd_i-r_it^rd_{m+i}$. So
\begin{equation*}
\mathcal{H}_N=\text{Span}_{\mathbb{C}}\{d_i,h(r)|i\in\{1,\dots,N\},r\in\mathbb{Z}^N\}
\end{equation*}
and has commutative Cartan subalgebra $\text{Span}_{\mathbb{C}}\{d_j|j\in\{1,\dots,N\}\}$. The Lie bracket in terms of homogeneous elements is given by $[d_i,h(r)]=r_ih(r)$  and
\begin{equation*}
[h(r),h(s)]=\sum_{i=1}^m(r_{m+i}s_i-r_is_{m+i})h(r+s).
\end{equation*}
Note that $h(0)=0$, and $[h(r),h(-r)]=0$.

Families of modules called category $\mathcal{J}$, have been defined for $\mathcal{W}_N$ and $\mathcal{S}_N$ in \cite{B} and \cite{BT} respectively. An analogous category of modules for $\mathcal{H}_N$ is defined as follows:
\begin{defn}\label{CatJ}
Let $N\geq2$ with $N$ even. An $\mathcal{H}_N$-module $J$ belongs to category $\mathcal{J}$ if the following properties hold:
\begin{enumerate}
\item[(J1)] The action of $d_i$ on $J$ is diagonalizable for all $i\in\{1,\dots,N\}$.
\item[(J2)] Module $J$ is a free $A_N$-module of finite rank.
\item[(J3)] For any $X\in\mathcal{H}_N,f\in A_N$ and $v\in J$, $X(fv)=(X f)v+f(X v).$
\end{enumerate}
\end{defn}
Any submodule of $J\in\mathcal{J}$ must be invariant under the actions of both $A_N$ and $\mathcal{H}_N$. By (J2) any module in $\mathcal{J}$ is a finite direct sum of indecomposable modules, hence the goal is to classify indecomposable modules $J\in\mathcal{J}$. Property (J1) decomposes $J$ into weight spaces, $J=\bigoplus_{\lambda\in\mathbb{C}^N}J_{\lambda}$ where $J_{\lambda}=\{v\in J|d_i v=\lambda_iv\}$. Note that $h(r)J_{\lambda}\subset J_{\lambda+r}$ since,
\begin{equation*}
d_i(h(r)v)=h(r)d_iv+[d_i,h(r)]v=(\lambda_i+r_i)h(r)v
\end{equation*}
for $v\in J_{\lambda}$. Similarly by (J3) $t^rJ_{\lambda}\subset J_{\lambda+r}$. These two relations partition the weights of $J$ into $\mathbb{Z}^N$-cosets of $\mathbb{C}^N$, and decompose $J$ into a direct sum of submodules, each corresponding to a distinct coset. Thus if $J$ is indecomposable its set of weights is one such coset $\lambda+\mathbb{Z}^N$ for $\lambda\in\mathbb{C}^N$ and $J=\bigoplus_{r\in\mathbb{Z}^N}J_{\lambda+r}$. Denote by $\mathcal{J}_{\lambda}$, the subcategory of $\mathcal{J}$ supported on $\lambda+\mathbb{Z}^N$ for a fixed $\lambda\in\mathbb{C}^N$ and from now on assume that $J\in\mathcal{J}_{\lambda}$.

Let $V=J_{\lambda}$. The invertible map $t^r:V\rightarrow J_{\lambda+r}$ identifies all weight spaces with $V$ and since $J$ is a free module for the associative algebra $A_N$ it follows that any basis for $V$ is also basis for $J$ viewed as a free $A_N$-module. The finite rank condition of (J2) implies that $V$ must be finite dimensional. This yields that $J\cong A_N\otimes V$. Homogeneous elements of $J$ will be denoted $t^s\otimes v$, for $s\in\mathbb{Z}^N,v\in V$.

The action of $h(r)$ on $J$ maps $1\otimes V\rightarrow t^r\otimes V$. This induces an endomorphism $H(r)$ on $V$ given by
\begin{equation*}
H(r)v=(t^{-r}\circ h(r))v.
\end{equation*}
The action in (J3) can then be expressed as
\begin{equation}\label{hdependsonH}
h(r)(t^s\otimes v)=\sum_{i=1}^m(r_{m+i}s_i-r_is_{m+i})t^{r+s}\otimes v+t^{r+s}\otimes H(r)v.
\end{equation}
Thus the action of $H(r)$ on $V$ determines the action of $h(r)$ on $J$.

The Lie bracket for $H(r)$ terms is
\begin{align*}
&\quad[H(r),H(s)]\\
&=[t^{-r}h(r),t^{-s}h(s)]\\
&=t^{-r}(h(r)(t^{-s}))h(s)-t^{-s}(h(s)(t^{-r}))h(r)+t^{-r-s}[h(r),h(s)]\\
&=t^{-r}\left(\sum_{i=1}^mr_{m+i}(-s_i)-r_i(-s_{m+i})\right)t^{r-s}h(s)\\
&\quad-t^{-s}\left(\sum_{i=1}^ms_{m+i}(-r_i)-s_i(-r_{m+i})\right)t^{s-r}h(r)\\
&\quad+t^{-r-s}\sum_{i=1}^m(r_{m+i}s_i-r_is_{m+i})h(r+s)\\
&=\sum_{i=1}^m(r_is_{m+i}-r_{m+i}s_i)(H(r)+H(s)-H(r+s)).
\end{align*}
\section{Polynomial Action}
The concept of a polynomial module was introduced in \cite{BB} and then generalized in \cite{BZ}. Polynomial modules were used in \cite{B} for classifying category $\mathcal{J}$ modules for $\mathcal{W}_N$. This technique was applied in the previous chapter to the $\mathcal{S}_N$ case and will be used again here. The main idea is to show that $H(r)$ acts on $V$ by an $\text{End}(V)$-valued polynomial in $r$ with constant term zero. Specifically,
\begin{prop}\label{PolynomialAction}
Let $N\geq 2$ (with $N$ even) and let $J=A_N\otimes V$ be an $\mathcal{H}_N$-module in category $\mathcal{J}$. Then $H(r)$ acts on $V$ by
\begin{equation*}
H(r)=\sum_{k\in\mathbb{Z}_{\geq0}^N}\frac{r^k}{k!}P^{(k)}-\delta_{r,0}P^{(0)},
\end{equation*}
for all $r\in\mathbb{Z}^N$, where $P^{(r)}\in\text{End}(V)$ and do not depend on $r$, and the summation is finite. 
\end{prop}
The Kronecker delta function is to ensure that $H(0)=0$. This will be proven by induction on $N$ with the base case $N=2$ already shown in \cite{BT} since $\mathcal{H}_2\cong\mathcal{S}_2$.
\begin{lem}\label{basecase}
Let $N=2$ and $J=A_2\otimes V$ be a module for $\mathcal{H}_2$ in category $\mathcal{J}$. Then $H(r)$ acts on $V$ by
\begin{equation*}
H(r)=\sum_{k\in\mathbb{Z}_{\geq0}^2}\frac{r^k}{k!}P^{(k)}-\delta_{r,0}P^{(0)},
\end{equation*}
for all $r\in\mathbb{Z}^2$, where $P^{(r)}\in\text{End}(V)$ and do not depend on $r$, and the summation is finite.
\end{lem}
See \cite{BT}, Proposition 3.7 for details, where it is shown that the action of $H(r)$ may be interpolated by an $\text{End}(V)$-valued polynomial in $r$ at all points in the plane except for the origin.

The fact that two (multivariate) polynomials are equal if they agree on a cube of arbitrary size was used repeatedly in \cite{BT} as it will be used here. The following Lemma will be required and is presented here without proof (cf. \cite{BT} Lemma 3.5).
\begin{lem}\label{equalpolynomials}
Let $S=S_1\times\dots\times S_N\in\mathbb{C}^N$, where each $S_i$ is a set with $K+1$ elements, and let $F$ and $G$ be polynomials of degree at most $K$ in $N$ variables, $X_1,\dots,X_N$, that agree on $S$. Then $F=G$.
\end{lem}
\begin{proof}[Proof of Proposition \ref{PolynomialAction}]
Proceed by induction on $N=2m$ with basis of induction given by Lemma \ref{basecase}. Let $\tilde{r}_i=r-r_ie_i-r_{m+i}e_{m+i}$ for $i\in\{1,\dots,m\}$. By induction hypothesis $H(\tilde{r}_i)$ acts by an $\text{End}(V)$-valued polynomial in $r$ for $\tilde{r}_i\neq0$. Let $k\in\{1,\dots,m\}$ with $k\neq i$. Then $H(e_i+e_k)$ and $H(e_i+e_{m+k})$ also act by an $\text{End}(V)$-valued polynomials in $r$ since $e_i+e_k$ and $e_i+e_{m+k}$ are constant. Then
\begin{equation*}
[H(\tilde{r}_i),H(e_i+e_k)]=-r_{m+k}(H(\tilde{r}_i)+H(e_i+e_k)-H(\tilde{r}_i+e_i+e_k)),
\end{equation*}
and
\begin{equation*}
[H(\tilde{r}_i),H(e_i+e_{m+k})]=r_k(H(\tilde{r}_i)+H(e_i+e_{m+k})-H(\tilde{r}_i+e_i+e_{m+k})).
\end{equation*}
Rearranging these equations gives
\begin{equation*}
r_{m+k}H(\tilde{r}_i+e_i+e_k)=r_{m+k}(H(\tilde{r}_i)+H(e_i+e_k))+[H(\tilde{r}_i),H(e_i+e_k)],
\end{equation*}
and
\begin{equation*}
r_kH(\tilde{r}_i+e_i+e_{m+k})=r_k(H(\tilde{r}_i)+H(e_i+e_{m+k}))-[H(\tilde{r}_i),H(e_i+e_{m+k})],
\end{equation*}
so that the terms on the right side act by an $\text{End}(V)$-valued polynomial in $r$ for $\tilde{r}_i\neq0$. Thus $r_{m+k}H(\tilde{r}_i+e_i+e_k)$ and $r_kH(\tilde{r}_i+e_i+e_{m+k})$ act by polynomial when $\tilde{r}_i\neq0$, and hence for all $r\in\mathbb{Z}^N$. It can be shown in a similar fashion that both $r_{m+k}H(\tilde{r}_i+e_{m+i}+e_k)$ and $r_kH(\tilde{r}_i+e_{m+i}+e_{m+k})$ act as polynomials for all $r$. This holds for all $i\neq k$, and in particular it follows that $r_iH(r_ie_i+r_{m+i}e_{m+i}+e_{m+i}+e_{m+k})$ acts by polynomial for all $r$. 

Now consider the Lie bracket
\begin{align*}
&[r_{m+k}H(\tilde{r}_i+e_i+e_k),r_iH(r_ie_i+r_{m+i}e_{m+i}+e_{m+i}+e_{m+k})]\\
&=r_{m+k}r_i(r_{m+i}+1)(r_k+1)\left(H(\tilde{r}_i+e_i+e_k)\right.\\
&\quad\left.+H(r_ie_i+r_{m+i}e_{m+i}+e_{m+i}+e_{m+k})-H(r+e_i+e_{m+i}+e_{k}+e_{m+k})\right).
\end{align*}
Isolating the term 
\begin{equation*}
r_{m+k}r_i(r_{m+i}+1)(r_k+1)H(r+e_i+e_{m+i}+e_{k}+e_{m+k})
\end{equation*}
on the right side shows that it must act by polynomial, call it $Q_1$, for all $r\in\mathbb{Z}^N$ since all other terms do. Similarly form the Lie brackets 
\begin{align*}
&[r_{k}H(\tilde{r}_i+e_i+e_{m+k}),r_iH(r_ie_i+r_{m+i}e_{m+i}+e_{m+i}+e_{k})],\\
&[r_{m+k}H(\tilde{r}_i+e_{m+i}+e_{k}),r_{m+i}H(r_ie_i+r_{m+i}e_{m+i}+e_{i}+e_{m+k})],\\
&[r_{k}H(\tilde{r}_i+e_{m+i}+e_{m+k}),r_{m+i}H(r_ie_i+r_{m+i}e_{m+i}+e_{i}+e_{k})]
\end{align*}
to show that
\begin{align*}
&r_{k}r_i(r_{m+i}+1)(r_{m+k}+1)H(r+e_i+e_{m+i}+e_{k}+e_{m+k}),\\
&r_{m+k}r_{m+i}(r_{i}+1)(r_k+1)H(r+e_i+e_{m+i}+e_{k}+e_{m+k}),\\
&r_{k}r_{m+i}(r_{i}+1)(r_{m+k}+1)H(r+e_i+e_{m+i}+e_{k}+e_{m+k})
\end{align*}
act by polynomials $Q_2,Q_3$, and $Q_4$ respectively, for all $r\in\mathbb{Z}^N$.

Note that 
\begin{equation*}
r_{k}(r_{m+k}+1)Q_1=r_{m+k}(r_{k}+1)Q_2
\end{equation*}
which implies that both $r_{m+k}$ and $r_{k}+1$ are factors of  $Q_1$ since they do not divide $r_{k}(r_{m+k}+1)$ (and both $r_{k}$ and $r_{m+k}+1$ are factors of $Q_2$). Comparing $Q_1$ and $Q_3$,
\begin{equation*}
r_{m+i}(r_{i}+1)Q_1=r_{i}(r_{m+i}+1)Q_3,
\end{equation*}
and so both $r_{i}$ and $r_{m+i}+1$ are factors of $Q_1$ since they do not divide $r_{m+i}(r_{i}+1)$. Thus
\begin{equation*}
Q_1=r_{m+k}r_i(r_{m+i}+1)(r_k+1)P_1
\end{equation*}
for some polynomial $P_1$, and hence $H(r+e_i+e_{m+i}+e_{k}+e_{m+k})=P_1$ for $r_i,r_{m+k}\neq0$ and $r_{m+i},r_k\neq-1$. By continuing this process of equating the $Q_i$ with appropriate linear factors, it can be shown that $H(r+e_i+e_{m+i}+e_{k}+e_{m+k})$ acts by a polynomial $P_2$ on the region $\{r_i,r_{k}\neq0\}\cap\{r_{m+i},r_{m+k}\neq-1\}$, $P_3$ on the region $\{r_{m+i},r_{m+k}\neq0\}\cap\{r_{i},r_{k}\neq-1\}$, and $P_4$ on the region $\{r_{m+i},r_{k}\neq0\}\cap\{r_{i},r_{m+k}\neq-1\}$.

Since these four regions are obtained by deleting a finite number of hyperplanes from $\mathbb{Z}^N$, their intersection contains a cube of arbitrary size. By Lemma \ref{equalpolynomials}, since $P_1, P_2, P_3$ and $P_4$ agree on this cube they must be equal, and so $H(r+e_i+e_{m+i}+e_{k}+e_{m+k})$ acts by polynomial on the union of the four regions, which is $\{(r_{i},r_{m+i})\neq(0,0),(-1,-1)\}\cap\{(r_{k},r_{m+k})\neq(0,0),(-1,-1)\}$.
Shifting coordinates, this implies that $H(r)$ acts by polynomial $P$ on the region
\begin{equation*}
\mathcal{R}=\{(r_{i},r_{m+i})\neq(0,0),(1,1)\}\cap\{(r_{k},r_{m+k})\neq(0,0),(1,1)\}.
\end{equation*}
What remains is to extend this region to all $r\in\mathbb{Z}^N\setminus\{0\}$.

Let $\mathcal{L}$ be the Lie algebra spanned by elements $H(r)$ for $r\in\mathbb{Z}^N$, and consider the automorphisms $\sigma_i:\mathcal{L}\rightarrow \mathcal{L}$ given by
\begin{equation*}
\sigma_i(H(r))=H(r_1,\dots,-r_i,\dots,-r_{m+i},\dots,r_N),
\end{equation*}
for $i\in\{1,\dots,m\}$. The region $\mathcal{R}$ may be expanded by considering automorphisms $\sigma_i$ of $\mathcal{L}$ since whatever was proven for $\mathcal{L}$ will also be true of its image under $\sigma_i$. Thus $\sigma_i(H(r))$ acts by polynomial $\bar{P}$ on $\mathcal{R}$ which implies that $H(r)$ acts by $\bar{P}$ on $\{(r_{i},r_{m+i})\neq(0,0),(-1,-1)\}\cap\{(r_{k},r_{m+k})\neq(0,0),(1,1)\}$. Lemma \ref{equalpolynomials} may be applied to the intersection of these regions, showing that $P=\bar{P}$ there, and hence $H(r)$ acts by $P$ on their union $\mathcal{R}'=\{(r_{i},r_{m+i})\neq(0,0)\}\cap\{(r_{k},r_{m+k})\neq(0,0),(1,1)\}$. In the same way, applying $\sigma_k$ to $\mathcal{R}$ yields that $H(r)$ acts by $P$ on the region $\mathcal{R}_{i,k}=\{(r_{i},r_{m+i})\neq(0,0)\}\cap\{(r_{k},r_{m+k})\neq(0,0)\}$.

Now apply Lemma \ref{equalpolynomials} to the intersection of regions $\mathcal{R}_{i,k}$ for various $k$ to show that $H(r)$ acts by $P$ on 
\begin{equation*}
\mathcal{R}_i=\bigcup_{\substack{k=1\\k\neq i}}^m\mathcal{R}_{i,k}=\{(r_i,r_{m+i})\neq(0,0)\}\cap\{\tilde{r}_i\neq0\}.
\end{equation*}
Applied the same lemma again to these regions shows that $H(r)$ acts by $P$ on 
\begin{equation*}
\bigcup_{i=1}^m\mathcal{R}_i=\bigcap_{i=1}^m\{\tilde{r}_i\neq0\}.
\end{equation*}

Fix $s\in\mathbb{Z}^N$ and let $\bar{s}_k=s_ke_k+s_{m+k}e_{m+k}$ for $k\in\{1,\dots,m\}$. Further assume that $\bar{s}_k\neq0$. Then
\begin{equation*}
[H(\bar{s}_k),H(r-\bar{s}_k)]=(s_kr_{m+k}-s_{m+k}r_k)(H(\bar{s}_k)+H(r-\bar{s}_k)-H(r))
\end{equation*}
and so
\begin{multline*}
(s_kr_{m+k}-s_{m+k}r_k)H(r)=\\
(s_kr_{m+k}-s_{m+k}r_k)(H(\bar{s}_k)+H(r-\bar{s}_k))-[H(\bar{s}_k),H(r-\bar{s}_k)].
\end{multline*}
From what was proven above, all operators on the right side of this equation act by polynomial on $\cap_{i=1}^m\{\tilde{r}_i\neq\bar{s}_k\}$, and hence $(s_kr_{m+k}-s_{m+k}r_k)H(r)$ acts by polynomial $T$ on this region. Similarly for $s'\in\mathbb{Z}^N$, $(s'_kr_{m+k}-s'_{m+k}r_k)H(r)$ acts by polynomial $T'$ on $\cap_{i=1}^m\{\tilde{r}_i\neq\bar{s}'_k\}$ and on the intersection of these regions
\begin{equation*}
(s'_kr_{m+k}-s'_{m+k}r_k)T=(s_kr_{m+k}-s_{m+k}r_k)T',
\end{equation*}
implying that $(s_kr_{m+k}-s_{m+k}r_k)$ is a factor of $T$. Thus $H(r)$ acts by on polynomial when $(s_kr_{m+k}-s_{m+k}r_k)\neq0$ and hence acts by polynomial on the region
\begin{equation*}
\mathcal{T}_{s,k}=\{s_kr_{m+k}\neq s_{m+k}r_k\}\cap\left(\bigcap_{i=1}^m\{\tilde{r}_i\neq\bar{s}_k\}\right).
\end{equation*}

Again use Lemma \ref{equalpolynomials} to patch together the regions which $H(r)$ acts by polynomial. First
\begin{equation*}
\left(\bigcap_{i=1}^m\{\tilde{r}_i\neq 0\}\right)\cup\mathcal{T}_{s,k}=\{(r_k,r_{m+k})\neq(0,0)\}\cap\{r\neq\bar{s}_k\}.
\end{equation*}
Then
\begin{equation*}
\bigcup_{s\in\mathbb{Z}^N}\left(\bigcap_{i=1}^m\{\tilde{r}_i\neq 0\}\right)\cup\mathcal{T}_{s,k}=\{(r_k,r_{m+k})\neq(0,0)\},
\end{equation*}
and
\begin{equation*}
\bigcup_{k=1}^m\{(r_k,r_{m+k})\neq(0,0)\}=\{r\neq0\}.
\end{equation*}
Thus $H(r)$ acts by $\text{End}(V)$-valued polynomial in $r$ for all $r\neq0$. To include $r=0$ and to accommodate the fact that $H(0)=0$, a delta function is added, and so $H(r)$ may be expressed
\begin{equation*}
H(r)=\sum_{k\in\mathbb{Z}_{\geq0}^N}\frac{r^k}{k!}P^{(k)}-\delta_{r,0}P^{(0)},
\end{equation*}
for all $r\in\mathbb{Z}^N$, where $P^{(r)}\in\text{End}(V)$, and the sum is finite.
\end{proof}
\section{Classification}	
The result of the previous section is that the action of the $H(r)$ operators are now in terms of the $P^{(k)}$, $k\in\mathbb{Z}_{\geq0}^N$. Following the ideas of \cite{B} and \cite{BT}, the Lie brackets between the operators $P^{(k)}$ are computed below and show that these operators realize a representation of a familiar Lie algebra. This will yield a classification of the modules in category $\mathcal{J}$.

To compute the Lie bracket $[P^{(j)},P^{(k)}]$ use Proposition \ref{PolynomialAction} and consider a region where $r,s\neq0$ so that
\begin{equation}\label{H(r)bracketPbracket}
[H(r),H(s)]=\sum_{j\in\mathbb{Z}_{\geq0}^N}\sum_{k\in\mathbb{Z}_{\geq0}^N}\frac{r^js^k}{j!k!}[P^{(j)},P^{(k)}].
\end{equation}
The left side of this evaluates to 
\begin{multline*}
\sum_{i=1}^m(r_is_{m+i}-r_{m+i}s_i)(H(r)+H(s)-H(r+s))\\=\sum_{i=1}^m(r_is_{m+i}-r_{m+i}s_i)\left(\sum_{j\in\mathbb{Z}_{\geq0}^N}\frac{r^j}{j!}P^{(j)}+\sum_{k\in\mathbb{Z}_{\geq0}^N}\frac{s^k}{k!}P^{(j)}-\sum_{\ell\in\mathbb{Z}_{\geq0}^N}\frac{(r+s)^{\ell}}{\ell!}P^{(\ell)}\right).
\end{multline*}
Then $[P^{(j)},P^{(k)}]$ is obtained by extracting the coefficient of $\frac{r^js^k}{j!k!}$ above for various $j$ and $k$. For $|j|,|k|>1$
\begin{equation*}
[P^{(j)},P^{(k)}]=\sum_{i=1}^m(j_{m+i}k_i-j_ik_{m+i})P^{(j+k-e_i-e_{m+i})}.
\end{equation*}
If $|j|\leq1$ then $[P^{(j)},P^{(k)}]=0$ for all $k$ with the exception that $[P^{(e_i)},P^{(e_{m+i})}]=P^{(0)}$ (and the usual anticommutativity property holds). 

Consider the Lie algebra of derivations of polynomials in $N$ variables,
\begin{equation*}
\text{Der}(\mathbb{C}[x_1,\dots,x_N])=\text{Span}_{\mathbb{C}}\left\{\left.x^r\frac{\partial}{\partial x_a}\right|a\in\{1,\dots,N\},r\in\mathbb{Z}_{\geq0}^N\right\}.
\end{equation*}
This has as a subalgebra the Hamiltonian Lie algebra given by
\begin{equation*}
\mathfrak{H}_N=\text{Span}_{\mathbb{C}}\left\{\left.X(r)=\sum_{i=1}^mr_{m+i}x^{r-e_{m+i}}\frac{\partial}{\partial x_i}-r_{i}x^{r-e_{i}}\frac{\partial}{\partial x_{m+i}}\right|r\in\mathbb{Z}_{\geq0}^N\right\},
\end{equation*}
with Lie bracket
\begin{equation*}
[X(r),X(s)]=\sum_{i=1}^m(r_{m+i}s_i-r_is_{m+i})X(r+s-e_i-e_{m+i}).
\end{equation*}
(Notation $\mathfrak{H}_N$ is used in place of $H_N$ from the introduction so as not to confuse it with operators $H(r)$).

Let $\mathfrak{H}_N^+=\text{Span}_{\mathbb{C}}\left\{X(r)\in\mathfrak{H}_N||r|>1\right\}$ and let $\mathfrak{h}_{N}$ be the Heisenberg algebra $\text{Span}_{\mathbb{C}}\left\{p^i,q^i,\mathfrak{c}|i\in\{1,\dots,m\}\right\}$ with $[p^i,q^i]=\mathfrak{c}$ and all other brackets zero. The brackets for $\mathcal{P}_N$ found above demonstrate the following.
\begin{prop}\label{PrepH}
The map $\rho:\mathfrak{H}_N^+\oplus\mathfrak{h}_{N}\rightarrow\text{End}(V)$ given by 
\begin{align*}
&\rho(X(r))=P^{(r)}\text{ for }|r|>1,\\
&\rho(p^i)=P^{(e_i)},\\
&\rho(q^i)=P^{(e_{m+i})},\\
&\rho(\mathfrak{c})=P^{(0)},
\end{align*}
is a finite dimensional representation of $\mathfrak{H}_N^+\oplus\mathfrak{h}_{N}$ on $V$ for $N=2m\geq2$.
\end{prop}
\begin{prop}\label{spmodules}
Let $\mathfrak{L}_n=\text{\normalfont{Span}}_{\mathbb{C}}\left\{\left.X(r)\in\mathfrak{H}_N\right||r|=n+2\right\}$. In the grading $\mathfrak{H}_N=\bigoplus_{n=-1}^{\infty}\mathfrak{L}_n$, component $\mathfrak{L}_0$ is isomorphic to the symplectic algebra $\mathfrak{sp}_N$ and each $\mathfrak{L}_n$ is an irreducible $\mathfrak{sp}_N$-module.
\end{prop}
\begin{proof}
Let $E_{ab}$ be the matrix with a $1$ in entry $(a,b)$ and zeros elsewhere. An isomorphism between $\mathfrak{L}_0$ and $\mathfrak{sp}_N$ is obtained by identifying basis elements  
\begin{equation*}
 x_i\frac{\partial}{\partial x_{j}}-x_{m+j}\frac{\partial}{\partial x_{m+i}}, x_i\frac{\partial}{\partial x_{m+j}}+x_{j}\frac{\partial}{\partial x_{m+i}}, x_{m+i}\frac{\partial}{\partial x_{j}}+x_{m+j}\frac{\partial}{\partial x_{i}},
\end{equation*}
of $\mathfrak{L}_0$ with $E_{i,j}-E_{m+j,m+i}, E_{i,m+j}+E_{j,m+i}, E_{m+i,j}+E_{m+j,i}$ respectively from $\mathfrak{sp}_N$, where $i,j\in\{1,\dots,m\}$. Then each $\mathfrak{L}_n$ is an $\mathfrak{sp}_N$-module via the adjoint action of $\mathfrak{L}_0$. The Cartan subalgebra of $\mathfrak{sp}_N$ is spanned by diagonal elements $ E_{i,i}-E_{m+i,m+i}$, and raising operators (positive root spaces) are spanned by  $E_{i,m+i},E_{i,j}-E_{m+j,m+i}, E_{i,m+j}+E_{j,m+i}$ for $i<j$.

Every finite dimensional simple $\mathfrak{sp}_N$-module is a highest weight module, and so by Weyl's Theorem on complete reducibility it suffices to show that there exists a unique (up to scalar) highest weight vector in each $\mathfrak{L}_i$. This requires showing that there is a unique $v\in\mathfrak{L}_n$ which is annihilated by the adjoint action of all raising operators. Elements $X(r)$ of $\mathfrak{L}_n$ are weight vectors as $\left[x_i\frac{\partial}{\partial x_{i}}-x_{m+i}\frac{\partial}{\partial x_{m+i}},X(r)\right]=(r_i-r_{m+i})X(r)$ for all $i$, and any scalar multiple of $X(ne_1)=-x_1^n\frac{\partial}{\partial x_{m+1}}$ is a highest weight vector of $\mathfrak{L}_n$ since for any $i<j$,
\begin{equation*}
\left[x_i\frac{\partial}{\partial x_{m+i}}, x_1^n\frac{\partial}{\partial x_{m+1}}\right]=0,
\end{equation*}
\begin{equation*}
\left[ x_i\frac{\partial}{\partial x_{j}}-x_{m+j}\frac{\partial}{\partial x_{m+i}}, x_1^n\frac{\partial}{\partial x_{m+1}}\right]=0,
\end{equation*}
\begin{equation*}
\left[ x_i\frac{\partial}{\partial x_{m+j}}+x_{j}\frac{\partial}{\partial x_{m+i}}, x_1^n\frac{\partial}{\partial x_{m+1}}\right]=0.
\end{equation*}
It remains to show that there are no other highest weight vectors.

A weight vector $u\in\mathfrak{L}_n$ of weight $\lambda$ can be expressed as a finite sum $u=\sum_{r\in I}C_rX(r)$ where $X(r)$ has weight $\lambda$ for each $r\in I$ and $C_r\in\mathbb{C}^N$. The task of finding highest weight vectors amounts to finding all such $u$ with $[X(s),u]=0$, for  $s=2e_i,e_i+e_{m+j}$ and $e_i+e_j$ with $i<j$. By linear independence
\begin{equation*}
0=[X(2e_i),u]=\sum_{r\in I}C_r(-2r_{m+i})X(r+e_i-e_{m+i})
\end{equation*}
implies that $C_r=0$ whenever $r_{m+i}\neq0$ for some $i$. Since $u$ is a weight vector this means that $u=X(r)$ for some $r$ with $r_{m+1}=\dots=r_{2m}=0$. Then
\begin{equation*}
0=[X(e_i+e_{m+j}),u]=r_jX(r+e_i-e_j)
\end{equation*}
implies $r_j=0$ for all $i,j$ with $i<j$. This leaves only the case $r=ne_1$ and thus $X(ne_1)$ is the unique highest weight vector of $\mathfrak{L}_n$ up to scalar. 
\end{proof}
Notice that $[X(e_i+e_{m+i}),X(ne_i)]=nX(ne_i)$ for $n\geq0$, suggesting that in the representation of $\mathfrak{sp}_N$ on $\mathfrak{H}_N^+$, there are infinitely many distinct eigenvalues. However the following Lemma shows that this cannot be the case, meaning that for the representation $\rho$ of Proposition \ref{PrepH}, there exists $k_0$ such that $\rho(X(ke_i))=0$  for all $k\geq k_0$. The fact that each $\mathfrak{L}_k$ is an irreducible $\mathfrak{sp}_N$-module implies that all of $\mathfrak{L}_k$ acts trivially on $V$ for $k\geq k_0$.
\begin{lem}\label{finiteeigenvals}
Let $\mathfrak{L}$ be a Lie algebra with nonzero elements $y,y_1,y_2,\dots$ with the property that
\begin{equation*}
[y,y_i]=\alpha_iy_i
\end{equation*}
for $i=1,2,\dots$, and $\alpha_i\in\mathbb{C}$. Then for a finite dimensional representation $(U,\rho)$ of $\mathfrak{L}$, there are at most $(\dim U)^2-\dim U+1$ distinct eigenvalues for which $\rho(y_i)\neq 0$.
\end{lem}
This Lemma and its proof are given in \cite{B}. It is also used in the proof of Lemma \ref{basecase} found in \cite{BT}.

With this the groundwork is laid for the main result.
\begin{thm}\label{classification}
Let $\lambda\in\mathbb{C}^N$ and let $\mathcal{J}_{\lambda}$ be a subcategory of modules in $\mathcal{J}$ supported on $\lambda+\mathbb{Z}^N$, where $N=2m\geq 2$. There is an equivalence of categories between the category of finite dimensional modules for $\mathfrak{H}_N^+\oplus\mathfrak{h}_{N}$ and $\mathcal{J}_{\lambda}$. This equivalence maps $V$ to $A_N\otimes V$ where $V$ is a finite dimensional module for $\mathfrak{H}_N^+\oplus\mathfrak{h}_{N}$. The action of $\mathcal{H}_N$ on $A_N\otimes V$ is given by $d_a(t^{s}\otimes v)=(s_a+\lambda_a)t^s\otimes v$,  and for $r\neq 0$,
\begin{multline}\label{h(r)action}
h(r)(t^{s}\otimes v)=\sum_{i=1}^m(r_{m+i}s_i-r_is_{m+i})t^{r+s}\otimes v+t^{r+s}\otimes \rho(\mathfrak{c})v\\
+t^{r+s}\otimes\sum_{i=1}^m \left(r_i\rho(p^i)+r_{m+i}\rho(q^i)\right)v+t^{r+s}\otimes\sum_{\substack{k\in\mathbb{Z}_{\geq 0}^N\\|k|>1}}\frac{r^k}{k!}\rho(X(k))v.
\end{multline}
\end{thm}
\begin{proof}
Following the definition of category $\mathcal{J}$, it was noted that each module in $\mathcal{J}_{\lambda}$ may be identified with $A_N\otimes V$, where $V$ is any weight space. Combining (J1) and (J3) gives the action for $d_i$, and applying Proposition \ref{PolynomialAction} to (\ref{hdependsonH}) yields
\begin{equation*}
h(r)(t^{s}\otimes v)=\sum_{i=1}^m(r_{m+i}s_i-r_is_{m+i})t^{r+s}\otimes v+t^{r+s}\otimes\sum_{k\in\mathbb{Z}_{\geq 0}^N}\frac{r^k}{k!}P^{(k)}v,
\end{equation*}
for $r\neq0$, and $h(0)=0$ otherwise. Proposition \ref{PrepH} shows that $V$ is a finite dimensional $\mathfrak{H}_N^+\oplus\mathfrak{h}_{N}$-module, and (\ref{h(r)action}) is obtained from the expression above.

On the other hand suppose $V$ is a finite dimensional module for $\mathfrak{H}_N^+\oplus\mathfrak{h}_{N}$. Identify the elements of $\mathfrak{H}_N^+\oplus\mathfrak{h}_{N}$ with the $P^{(k)}$ as in Proposition \ref{PrepH} and let $H(r)=\sum_{k\in\mathbb{Z}_{\geq 0}^N}\frac{r^k}{k!}P^{(k)}$. This sum has only finitely many nonzero terms because in the decomposition $\mathfrak{H}_N^+=\bigoplus_{k=0}^{\infty}\mathfrak{L}_k$, there exits $k_0$ such that $\mathfrak{L}_k$ acts trivially on $V$ for all $k\geq k_0$. The Lie bracket for $H(r)$ terms is obtained from (\ref{H(r)bracketPbracket}), and defining the action of $h(r)$ on $A_N\otimes V$ by (\ref{hdependsonH}) recovers the commutator relations for $\mathcal{H}_N$. Thus $A_N\otimes V$ is a finite dimensional $\mathcal{H}_N$-module.
\end{proof}
\section{Irreducible Representations}
The action of $\mathcal{H}_N$ simplifies considerably when only simple modules from category $\mathcal{J}$ are considered. The central element of the Heisenberg algebra acts trivially in a finite dimensional irreducible representation, and so an $N$ dimensional abelian algebra $\mathfrak{a}_N=\text{Span}_{\mathbb{C}}\left\{C_i|i\in\{1,\dots,N\}\right\}$ replaces $\mathfrak{h}_{N}$.

The following Lemma will be used to show that the action of $\mathfrak{H}_N^+$ reduces to that of $\mathfrak{sp}_N$. See Lemma 2.4 in \cite{CS} for the proof.
\begin{lem}[\cite{CS}, Lemma 2.4]\label{trivialaction}
Let $\mathfrak{g}$ be a finite dimensional Lie algebra over $\mathbb{C}$ with solvable radical $\text{\emph{Rad}}(\mathfrak{g})$. Then $[\mathfrak{g},\text{\emph{Rad}}(\mathfrak{g})]$ acts trivially on any finite dimensional irreducible $\mathfrak{g}$-module.
\end{lem}
To apply this lemma let $\mathfrak{g}=\mathfrak{H}_N^+/I\oplus\mathfrak{a}_N$, where $I$ is the ideal $I=\bigoplus_{k\geq k_0}\mathfrak{L}_k$ which acts trivially on $V$ for some $k_0$. Then $\text{Rad}(\mathfrak{g})=(\bigoplus_{k>0}\mathfrak{L}_k)/I\oplus\mathfrak{a}_N$ since $[\mathfrak{L}_n,\mathfrak{L}_k]\subset\mathfrak{L}_{n+k}$. So $[\mathfrak{g},\text{Rad}(\mathfrak{g})]=(\bigoplus_{k>0}\mathfrak{L}_k)/I$ acts trivially by Lemma \ref{trivialaction}, and hence the ideal $\bigoplus_{k>0}\mathfrak{L}_k$ must act trivially on $V$.
\begin{thm}\label{irreducibleclassification}
Let $\lambda\in\mathbb{C}^N$ and let $\mathcal{J}_{\lambda}$ be a subcategory of modules in $\mathcal{J}$ supported on $\lambda+\mathbb{Z}^N$. For $N=2m\geq 2$, there is a one-to-one correspondence between the finite dimensional irreducible modules for $\mathfrak{sp}_N\oplus\mathfrak{a}_N$ and the irreducible modules in $\mathcal{J}_{\lambda}$. This correspondence maps a finite dimensional irreducible module $V$  for $\mathfrak{sp}_N\oplus\mathfrak{a}_N$ to $A_N\otimes V$. The action of $\mathcal{H}_N$ on $A_N\otimes V$ is given by $d_a(t^{s}\otimes u)=(s_a+\lambda_a)t^s\otimes v$ and for all $r$
\begin{multline*}
h(r)(t^{s}\otimes v)=\sum_{i=1}^m(r_{m+i}(s_i+\mu_{m+i})-r_i(s_{m+i}+\mu_i))t^{r+s}\otimes v\\
+t^{r+s}\otimes\sum_{i=1}^m\left(r_{m+i}^2\varphi(E_{m+i,i})+r_ir_{m+i}\varphi(E_{i,i}-E_{m+i,m+i})-r_i^2\varphi(E_{i,m+i})\right)v\\
+t^{r+s}\otimes\sum_{\substack{i,j=1\\i< j}}^m\left(r_{m+i}r_{m+j}\varphi(E_{m+j,i}+E_{m+i,j})+r_ir_{m+j}\varphi(E_{i,j}-E_{m+i,m+j})\right.\\
\left.-r_ir_j\varphi(E_{i,m+j}+E_{j,m+i})\right)v,
\end{multline*}
where $\mu_i$ is the action of $C_i\in\mathfrak{a}_N$, and $\varphi$ is an irreducible representation of $\mathfrak{sp}_N$.
\end{thm}
\begin{proof}
Since $\bigoplus_{k>0}\mathfrak{L}_k$ acts trivially on $V$, the action of $\mathfrak{H}_N^+$ reduces to that of $\mathfrak{L}_0\cong \mathfrak{sp}_N$. As mentioned above, the Heisenberg algebra is replaced by abelian algebra $\mathfrak{a}_N$, and by Schur's Lemma its elements $C_i$ act by scalars $\mu_i$ in an irreducible representation. Thus the action in (\ref{h(r)action}) reduces to that above. Theorem \ref{classification} gives the correspondence between these sets of modules. 
\end{proof}
\section{Acknowledgements}
Many thanks to Professor Yuly Billig for his generous help on this paper.



\begin{thebibliography}{8}
\bibitem{BB} S. Berman, Y.Billig (1999), Irreducible representations for toroidal Lie algebras, Journal of Algebra 221(1), 188-231. doi:10.1006/jabr.1999.7961.
\bibitem{B} Y. Billig (2007), Jet Modules, Canadian Journal of Mathematics, 59(4), 712-729. doi: 10.4153/CJM-2007-031-2.
\bibitem{BT} Y. Billig and J. Talboom (2016), Classification of Category $\mathcal{J}$ modules for Divergence Zero Vector Fields on a Torus, arXiv:1607.07067.
\bibitem{BZ} Y. Billig, K. Zhao (2004), Weight modules over exp-polynomial Lie algebras, Journal of Pure and Applied Algebra 191(1), 23-42. doi: 10.1016/j.jpaa.2003.12.004.
\bibitem{CS} L. Cagliero, F. Szechtman (2011), Jordan-Chevalley Decomposition in Finite Dimensional Lie Algebras, Proceedings of the American Mathematical Society, 139(11), 309-3913. doi: 10.1090/S0002-9939-2011-10827-X.
\bibitem{JL} J.-J. Jiang, W.-Q. Lin (2015), Partial classification of cuspidal simple modules for Virasoro-like algebra, arXiv:1508.05743.
\bibitem{Rao2} S. Eswara Rao (2004), Partial classification of modules for Lie algebra of diffeomorphisms of d-dimensional torus, Journal of Mathematical Physics, 45 , 3322-3333. doi: 10.1063/1.1769104.
\bibitem{Ru} A. N. Rudakov (1974), Irreducible representations of infinite-dimensional Lie algebras of Cartan type, Math. USSR Izvestija, 8(4), 836-866. doi:10.1070/IM1974v008n04ABEH002129.
\end{thebibliography}
\end{document}